\documentclass[preprint,12pt,linenumbers]{elsarticle}
\usepackage[mathscr]{eucal}
\usepackage{graphics,pstricks,pst-node,clrscode,longtable,pstricks-add}
\usepackage{clrscode}
\usepackage{pbsi}
\usepackage{amsthm}
\theoremstyle{plain}
\newtheorem{theorem}{Theorem}[section]
\newtheorem{lemma}[theorem]{Lemma}

\newtheorem{corollary}[theorem]{Corollary}
\theoremstyle{definition}
\newtheorem{definition}[theorem]{Definition}

\newcommand{\gra}[1]{{\mathbf{#1}}}  
\newcommand{\defn}[1]{\textit{#1}}

\usepackage{amssymb}

\begin{document}

\begin{frontmatter}

\title{Counting the Number of Minimal Paths in Weighted Coloured--Edge Graphs}
\author{Andrew Ensor\corref{cor1}}
\address{School of Computing and Mathematical Sciences, Auckland University of Technology, %
2-14 Wakefield St, Auckland 1142, New Zealand.}
\ead{andrew.ensor@aut.ac.nz}
\author{Felipe Lillo\corref{cor2}}
\ead{flillo@ucm.cl}
\address{School of Management Sciences, Universidad Cat\'{o}lica del Maule, %
3605 Avenida San Miguel, Talca,
Chile.}
\cortext[cor1]{Tel:+64 9 921 9999  extn:8485}
\cortext[cor2]{Tel:+56 71 203 100  extn:533}

\begin{abstract}
A weighted coloured--edge graph is a graph for which each edge is assigned both a
positive weight and a discrete colour, and can be used to model transportation
and computer networks in which there are multiple transportation modes.
In such a graph paths are compared by their total weight in each colour,
resulting in a Pareto set of minimal paths from one vertex to another.
This paper will give a tight upper bound on the cardinality of a minimal set of paths
for any weighted coloured--edge graph.
Additionally, a bound is presented on the expected number of minimal paths
in weighted bicoloured--edge graphs.
\end{abstract}

\begin{keyword}
graph theory \sep minimal paths \sep multimodal network
\sep transportation modes \sep weighted coloured--edge graph


\MSC[2010] 05C38 \sep 05C22

\end{keyword}

\end{frontmatter}


\section{Introduction}\label{Sec:Introduction}

\begin{definition}
A \defn{weighted coloured--edge graph} $\gra{G}=\langle V, E, \omega, \lambda \rangle$
consists of a directed multigraph $\langle V, E \rangle$ with vertex set $V$ and edge set
$E$, a \defn{weight} function $\omega\colon E\rightarrow \mathbb{R}^+$,
and a (surjective) \defn{colour} function $\lambda\colon E\rightarrow M$,
where $M$ is a set of possible colours for the edges.
\end{definition}

Hence associated with each edge $e\in E$, there is an initial vertex $u\in V$ and a
terminal vertex $v\in V$, a positive weight $\omega(e)\in \mathbb{R}^+$,
and a colour $\lambda(e)\in M$.
The graph $\gra{G}$ is said to be \defn{finite} if both $V$ and $E$ are finite sets,
in which case $M$ is also finite.

Concepts similar to weighted coloured--edge graphs have received little attention
in the literature. Cl\'{i}maco et al. \cite{Climaco2010} experimentally studied
the number of spanning trees in a weighted graph whose edges are labelled with
a colour. In that work, weight and colour are two criteria both to be minimized and the
proposed algorithm generates a set of non--dominated spanning trees. The computation
of coloured paths in a weighted coloured--edge graph is investigated by Xu et al. \cite{Xu2009}.
The main feature of their approach is a graph reduction technique based on a priority
rule. This rule basically transforms a weighted coloured--edge multidigraph into a
coloured--vertex digraph by applying algebraic operations to the
adjacency matrix. Additionally, the authors provide an algorithm to identify coloured
source--destination paths. Nevertheless, the algorithm is not intended
for general instances because its input is a unit--weighted coloured multidigraph and only
paths not having consecutive edges equally coloured are considered.
In his paper Manoussakis \cite{Manoussakis1995} studied the computation of
paths with specific colour patterns in unweighted coloured--edge graphs. Particularly,
his study focuses on alternating coloured--edge paths in complete coloured--edge
graphs. Finally, Bang \& Gutin \cite{BangJensen199739} presents a survey of
the computation of alternating cycles and paths in coloured--edge multigraphs.
Manoussakis as well as Bang \& Gutin only focus on unweighted graphs.

As this paper investigates shortest paths and all weights are positive,
it is presumed that graphs do not have any self-loops.
Similarly, it can be presumed that for any two vertices $x,y\in V$ and colour $c\in M$
there is at most one edge $e_{xy}\in E$ from $u$ to $v$ for which $\lambda(e_{xy})=c$.
Therefore, for a finite graph with $n=\left|V\right|$,
$m=\left|E\right|$, $k=\left|M\right|$,
there is a bound on the number of edges, given by $m\leq kn(n-1)$.

A \defn{path} in $\gra{G}$ consists of a finite set of edges
$\left\{e_{x_0 x_1},e_{x_1 x_2},\dots,e_{x_{l-1} x_l}\right\}$ for which the terminal vertex
$x_i$ of each $e_{x_{i-1} x_i}$ is the initial vertex of $e_{x_i x_{i+1}}$,
and the path is considered to be from the source $x_0$ to the destination $x_l$.
The path is called \defn{simple} if no two edges in the path have the same initial vertex
nor the same terminal vertex.
It is straightforward to verify that a finite graph has at most the following
number of simple paths from a chosen source vertex to a destination vertex:
\begin{displaymath}
  k + k^2(n-2) + k^3(n-2)(n-3) + \dots + k^{n-1}(n-2)(n-3)\cdots 1.
\end{displaymath}
For any path $p_{uv}=\left\{e_{x_0 x_1},e_{x_1 x_2},\dots,e_{x_{l-1} x_l}\right\}$
from a vertex $u=x_0$ to a vertex $v=x_l$ and any colour $c\in M$
the path weight in colour $c$ is defined by:
\begin{displaymath}
  \omega_c(p_{uv}) = \sum_{\lambda(e_{x_i x_{i+1}})=c} \omega(e_{x_i x_{i+1}}),
\end{displaymath}
namely the sum of the weights for those edges that have colour $c$.
The \defn{weight} of a path is represented as a $k$-tuple
$\left(\omega_{c_1}(p_{uv}),\ldots,\omega_{c_i}(p_{uv}),\ldots,\omega_{c_k}(p_{uv})\right)$,
giving the total weight of the path in each colour.
A preorder $\leq$ can be defined on the paths from $u$ to $v$ by
$p_{uv}\leq q_{uv}$ if for every colour $c$ one has $\omega_c(p_{uv})\leq\omega_c(q_{uv})$,
essentially using the partial order defined on the weights by the product
partial ordering on $\mathbb{R}^k$.

From the above definitions it is apparent that the concept of
a weighted coloured--edge graph with $k$ colours can equivalently be formulated
as a multiweighted multigraph where each edge is assigned a $k$-tuple of non-negative
weights $\left(w_{c_1},\dots,w_{c_i},\dots,w_{c_k}\right)$ and exactly one $w_{c_i}>0$.
However, multiweighted graphs are mostly used in multicriteria optimization applications
where the weight components correspond to quantities to be optimized, such as cost and time,
so edges typically contribute toward more than just one quantity.
For this reason multiweighted graphs whose edge weights are zero in all but one component
have not received attention in the literature.
Furthermore, most multicriteria optimization techniques do not consider multigraphs,
although some can be generalized to include such possibilities.
By contrast, weighted coloured--edge graphs are suitable for
multimodal network applications,
which can be modeled using weighted coloured--edge graphs
where the weights represent some form of distance in a network
and the colours represent the mode of transportation.
In the past most multimodal network optimization applications have needed constraints
to be placed on the network to make the determination of a shortest path fulfilling
application-specific criteria tractable.
Often such criteria can be specified based on the total path weights in each mode,
so if the set of minimal paths has manageable cardinality then the
application-specific criteria might only need be
applied to the set of minimal paths rather than to potentially $O\left(k^{n-1}(n-2)!\right)$ paths.

This paper is interested in establishing bounds on the number of paths that can be minimal
from one vertex $u$ to another vertex $v$.
Certainly, if the graph is disconnected then there might be no paths
from $u$ to $v$, and if it is connected with edges in each allowed colour
then there are at least $k$ minimal paths, since paths that have edges solely in
single but distinct colours are incomparable.
However, an upper bound might appear to be difficult to establish.

As a simple example, consider the following graph with three vertices,
six edges, and two possible colours $M=\left\{\mbox{red}, \mbox{green}\right\}$
where red edges are indicated by dashed arcs.

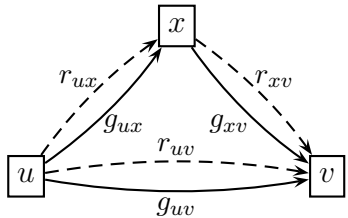
\begin{center}
\begin{pspicture}[shift=-1](0,-0.3)(4.5,2.5)
  \rput(0.25,0.25){\rnode{u}{\psframebox{$u$\rule[-0.75mm]{0mm}{3mm}}}}
  \rput(2.25,2.25){\rnode{x}{\psframebox{$x$\rule[-0.75mm]{0mm}{3mm}}}}
  \rput(4.25,0.25){\rnode{v}{\psframebox{$v$\rule[-0.75mm]{0mm}{3mm}}}}
  \psset{arrows=->,arrowsize=1.5mm}
  \nccurve[angleA=55,angleB=215,linestyle=dashed]{u}{x}
    \rput(0.95,1.55){\small$r_{ux}$}
  \nccurve[angleA=35,angleB=235]{u}{x}
    \rput(1.55,0.95){\small$g_{ux}$}
  \nccurve[angleA=10,angleB=170,linestyle=dashed]{u}{v}
    \rput(2.25,0.65){\small$r_{uv}$}
  \nccurve[angleA=350,angleB=190]{u}{v}
    \rput(2.25,-0.15){\small$g_{uv}$}
  \nccurve[angleA=325,angleB=125,linestyle=dashed]{x}{v}
    \rput(3.55,1.55){\small$r_{xv}$}
  \nccurve[angleA=305,angleB=145]{x}{v}
    \rput(2.95,0.95){\small$g_{xv}$}
\end{pspicture}
\end{center}

\noindent
This graph has six paths from $u$ to $v$ with the following weights:
\begin{center}
  $\left( {r_{ux}\!+\!r_{xv}}, {0} \right)$,
  $\left( {r_{ux}}, {g_{xv}} \right)$,
  $\left( {r_{xv}}, {g_{ux}} \right)$,
  $\left( {0}, {g_{ux}\!+\!g_{xv}} \right)$,
  $\left( {r_{uv}}, {0} \right)$,
  $\left( {0}, {g_{uv}} \right)$.
\end{center}

\noindent
Certainly the path weights
$\left( {r_{ux}\!+\!r_{xv}}, {0} \right)$
and $\left( {r_{uv}}, {0} \right)$
must be comparable under the product partial order on $\mathbb{R}^2$, as are
$\left( {0}, {g_{ux}\!+\!g_{xv}} \right)$
and $\left( {0}, {g_{uv}} \right)$.
Depending on the six edge weights there can be either two, three, or four
minimal paths from $u$ to $v$ with distinct weights.

Similarly, for the following graph with four vertices, fourteen edges
and two possible colours:

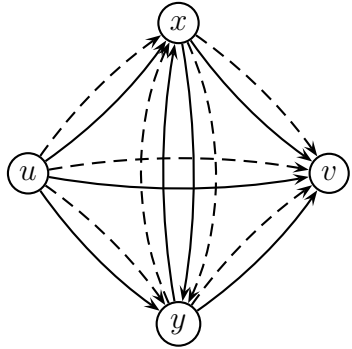
\begin{center}
\begin{pspicture}(0,-2)(4.5,2.5)
  \rput(0.25,0.25){\circlenode{u}{{$u$}}}
  \rput(2.25,2.25){\circlenode{x}{{$x$}}}
  \rput(2.25,-1.75){\circlenode{y}{$y$}}
  \rput(4.25,0.25){\circlenode{v}{$v$}}
  \psset{arrows=->,arrowsize=1.5mm}
  \nccurve[angleA=55,angleB=215,linestyle=dashed]{u}{x}
  \nccurve[angleA=35,angleB=235]{u}{x}
  \nccurve[angleA=10,angleB=170,linestyle=dashed]{u}{v}
  \nccurve[angleA=350,angleB=190]{u}{v}
  \nccurve[angleA=325,angleB=125,linestyle=dashed]{u}{y}
  \nccurve[angleA=305,angleB=145]{u}{y}
  \nccurve[angleA=325,angleB=125,linestyle=dashed]{x}{v}
  \nccurve[angleA=305,angleB=145]{x}{v}
  \nccurve[angleA=55,angleB=215,linestyle=dashed]{y}{v}
  \nccurve[angleA=35,angleB=235]{y}{v}
  \nccurve[angleA=295,angleB=65,linestyle=dashed]{x}{y}
  \nccurve[angleA=280,angleB=80]{x}{y}
  \nccurve[angleA=115,angleB=245,linestyle=dashed]{y}{x}
  \nccurve[angleA=100,angleB=260]{y}{x}
\end{pspicture}
\end{center}

\noindent
there are $26$ paths to consider from $u$ to $v$.
After a surprising amount of effort it can be seen that there are between
two and a maximum of eight minimal path weights possible.

As an important special case, a weighted coloured--edge graph
is called a \defn{chain} if in some enumeration
of its vertices $v_1, v_2, v_3, \dots, v_{n-1}, v_n$ the graph only has edges
from a vertex $v_x$ to the next vertex $v_{x+1}$ in the enumeration.

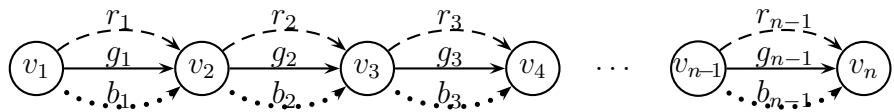
\begin{pspicture}(0,-1)(12,1)
  \rput(0.5,0){\circlenode{v1}{\makebox[4mm]{$v_1$}}}
  \rput(2.7,0){\circlenode{v2}{\makebox[4mm]{$v_2$}}}
  \rput(4.9,0){\circlenode{v3}{\makebox[4mm]{$v_3$}}}
  \rput(7.1,0){\circlenode{v4}{\makebox[4mm]{$v_4$}}}
  \rput(9.3,0){\circlenode{vnm1}{\makebox[4mm]{$v_{n\!-\!1}$}}}
  \rput(11.5,0){\circlenode{vn}{\makebox[4mm]{$v_n$}}}

  \psset{arrows=->,arrowsize=1.5mm}
  \nccurve[angleA=40,angleB=140,linestyle=dashed]{v1}{v2}\rput(1.6,0.65){$r_1$}
  \ncline{v1}{v2}\rput(1.6,0.15){$g_1$}
  \nccurve[angleA=320,angleB=220,linestyle=dotted,linewidth=2pt]{v1}{v2}\rput(1.6,-0.35){$b_1$}

  \nccurve[angleA=40,angleB=140,linestyle=dashed]{v2}{v3}\rput(3.8,0.65){$r_2$}
  \ncline{v2}{v3}\rput(3.8,0.15){$g_2$}
  \nccurve[angleA=320,angleB=220,linestyle=dotted,linewidth=2pt]{v2}{v3}\rput(3.8,-0.35){$b_2$}

  \nccurve[angleA=40,angleB=140,linestyle=dashed]{v3}{v4}\rput(6,0.65){$r_3$}
  \ncline{v3}{v4}\rput(6,0.15){$g_3$}
  \nccurve[angleA=320,angleB=220,linestyle=dotted,linewidth=2pt]{v3}{v4}\rput(6,-0.35){$b_3$}

  \rput(8.2,0){\dots}

  \nccurve[angleA=40,angleB=140,linestyle=dashed]{vnm1}{vn}\rput(10.45,0.65){$r_{n-1}$}
  \ncline{vnm1}{vn}\rput(10.45,0.15){$g_{n-1}$}
  \nccurve[angleA=320,angleB=220,linestyle=dotted,linewidth=2pt]{vnm1}{vn}\rput(10.45,-0.35){$b_{n-1}$}
\end{pspicture}

\noindent
Clearly, a chain can have up to a maximum of $k^{n-1}$ paths from $v_1$ to $v_n$.
A simple induction argument can be used to show that in the following chain
all $k^{n-1}$ paths have minimal weight:

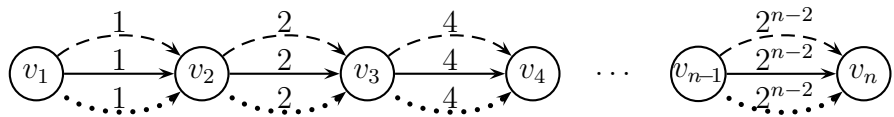
\begin{pspicture}(0,-1)(12,1)
  \rput(0.5,0){\circlenode{v1}{\makebox[4mm]{$v_1$}}}
  \rput(2.7,0){\circlenode{v2}{\makebox[4mm]{$v_2$}}}
  \rput(4.9,0){\circlenode{v3}{\makebox[4mm]{$v_3$}}}
  \rput(7.1,0){\circlenode{v4}{\makebox[4mm]{$v_4$}}}
  \rput(9.3,0){\circlenode{vnm1}{\makebox[4mm]{$v_{n\!-\!1}$}}}
  \rput(11.5,0){\circlenode{vn}{\makebox[4mm]{$v_n$}}}

  \psset{arrows=->,arrowsize=1.5mm}
  \nccurve[angleA=40,angleB=140,linestyle=dashed]{v1}{v2}
  \ncline{v1}{v2}
  \nccurve[angleA=320,angleB=220,linestyle=dotted,linewidth=2pt]{v1}{v2}

  \nccurve[angleA=40,angleB=140,linestyle=dashed]{v2}{v3}
  \ncline{v2}{v3}
  \nccurve[angleA=320,angleB=220,linestyle=dotted,linewidth=2pt]{v2}{v3}

  \nccurve[angleA=40,angleB=140,linestyle=dashed]{v3}{v4}
  \ncline{v3}{v4}
  \nccurve[angleA=320,angleB=220,linestyle=dotted,linewidth=2pt]{v3}{v4}

  \rput(8.2,0){\dots}

  \nccurve[angleA=40,angleB=140,linestyle=dashed]{vnm1}{vn}
  \ncline{vnm1}{vn}
  \nccurve[angleA=320,angleB=220,linestyle=dotted,linewidth=2pt]{vnm1}{vn}

  \rput[b](1.6,0.55){$1$}\rput[b](1.6,0.05){$1$}\rput[b](1.6,-0.45){$1$}
  \rput[b](3.8,0.55){$2$}\rput[b](3.8,0.05){$2$}\rput[b](3.8,-0.45){$2$}
  \rput[b](6.0,0.55){$4$}\rput[b](6.0,0.05){$4$}\rput[b](6.0,-0.45){$4$}
  \rput[b](10.45,0.55){$2^{n-2}$}\rput[b](10.45,0.05){$2^{n-2}$}
    \rput[b](10.45,-0.45){$2^{n-2}$}
\end{pspicture}

\noindent
Adding more \defn{forward} edges to the chain from $x$ to $y$ for $y>x+1$ increases
the number of paths in the graph but it cannot increase the number in a minimal
set of incomparable paths, and might possibly decrease the number.
But adding \defn{backward} edges from $x$ to $y$ for $y<x$ appears
to greatly complicate the situation.
However, it is shown in this paper that essentially chains illustrate
the worst possible situation, whereas more general weighted coloured--edge graphs
may have a factorial number of paths a minimal set of incomparable paths
can only have cardinality up to $k^{n-1}$ (hence in a multimodal network this
is a tight bound on the Pareto set of minimal path weights that might need to
be considered when applying application-specific constraints).
It is interesting to note that an equivalent result for general multiweighted graphs
(even restricting attention to only consider unigraphs) does not hold.

This paper also addresses the determination of a bound on the expected number
of minimal paths when the edge weights are random variables
in a weighted bicoloured--edge graph ($k=2$).
It provides an $O(n^3)$ probabilistic bound
on the expected number of minimal paths for bicoloured--edge graphs whose
weights are drawn from a bounded probability density function.

\section{Upper Bound on Minimal Paths}\label{Sec:UpperBound}

To prove that $k^{n-1}$ is a bound on the cardinality a special class of
weighted coloured--edge graph is first introduced.

\begin{definition}
A weighted coloured--edge graph $\gra{G}=\langle V, E, \omega, \lambda \rangle$
is called \defn{canonical} if:
\begin{itemize}
  \item $\gra{G}$ is \defn{complete} in each colour,
namely for all vertices $x\neq y$ and colour $c$
there is exactly one edge $e_{xy}$ from $x$ to $y$ with
$\lambda\left(e_{xy}\right)=c$,
  \item $\gra{G}$ satisfies the \defn{triangle inequality} in each colour,
for all distinct vertices $x$, $y$, $z$ and colour $c$,
the triangle formed by the three edges $e_{xy}$, $e_{yz}$, $e_{xz}$ with
$\lambda\left(e_{xy}\right)\!=\!\lambda\left(e_{yz}\right)\!%
=\!\lambda\left(e_{xz}\right)=c$
obeys
$\omega\left(e_{xz}\right)\leq\omega\left(e_{xy}\right)+\omega\left(e_{yz}\right)$.
\end{itemize}
\end{definition}

It is not difficult to verify in a canonical graph that every edge
$e_{xy}$ gives a minimal path from $x$ to $y$.
The following lemma shows that it will be sufficient to establish the bound on
the class of canonical graphs.

\begin{lemma}\label{Lem:CanonicalRepresentation}
Given any finite weighted coloured--edge graph
$\gra{G}=\langle V, E, \omega, \lambda \rangle$ there
is a canonical graph $\gra{G}^*=\langle V, E^*, \omega^*, \lambda^* \rangle$
with the same vertices and colours as $\gra{G}$,
where $E\subseteq E^*$, $\left.\lambda^*\right|_E=\lambda$,
and for which every minimal path in $\gra{G}$ is also minimal in $\gra{G}^*$.
\end{lemma}

\begin{proof}
Firstly note that $\gra{G}$ can be made complete by adding edges $e_{xy}$ with
weight $n\cdot w$ where $n=\left|V\right|$ and $w$ is the maximum weight of
any edge in $\gra{G}$.
The added edges won't affect any existing minimal paths as those paths contain
at most $n-1$ edges so their path weight in each colour is less than $n\cdot w$.
It might however introduce additional minimal paths in the graph
if there were no existing path from $x$ to $y$ in some colour.
Take $E^*$ to be the resulting set of edges.

Next, the graph can have its weights altered by defining
$\omega^*\left(e_{xy}\right)$ to be the weight of the shortest path from $x$ to $y$
that only uses edges with colour $\lambda^*\left(e_{xy}\right)$.
Thanks to completeness $\omega^*$ is well-defined and clearly the resulting graph
satisfies the triangle inequality.
\end{proof}

\begin{theorem}\label{Thm:IncomparableMinimalPaths}
Suppose $\gra{G}$ is a weighted coloured--edge graph with $n\geq2$ vertices and $k$ colours.
Then a set of incomparable minimal paths in $\gra{G}$ from one vertex to another
can have at most $k^{n-1}$ paths.
\end{theorem}

\begin{proof}
The proof uses a counting argument and induction to bound the cardinality $f_c(n)$
of a set of incomparable minimal paths whose first edge has colour $c$ in
any weighted coloured--edge graph with $n$ vertices.
Trivially, $f_c(2)=1$ for any graph $\gra{G}$ with only two vertices.

For the inductive step presume that $f_c(m)\leq k^{m-2}$ in any graph
with $m\leq n$ vertices and suppose $\gra{G}=\langle V, E, \omega, \lambda \rangle$
has $n+1$ vertices.
By Lemma \ref{Lem:CanonicalRepresentation} $\gra{G}$
can be presumed to be canonical.
Let $u$ and $v$ be any two distinct vertices of $\gra{G}$ and $S$ be a set of
incomparable minimal paths from $u$ to $v$.
By the triangle inequality it can be presumed that no minimal path in $S$
has two consecutive edges of the same colour.
Furthermore, a useful observation for a minimal path that starts with an edge
$e_{ux}$ of colour $\lambda\left(e_{ux}\right)=c$ and which passes through
some vertex $y$ before reaching $v$ is that for the edge $e_{uy}$ with
$\lambda\left(e_{uy}\right)=c$ minimality of the path ensures that
$\omega\left(e_{ux}\right)<\omega\left(e_{uy}\right)$.
To show that $f_c(n+1)\leq k^{n-1}$ for each colour $c$
order the remaining vertices of $V$, $v_1, v_2, \dots, v_{n-1}$
so that if $i<j$ then $\omega\left(e_{u v_i}\right)\leq\omega\left(e_{u v_j}\right)$
where $e_{u v_i}$ and $e_{u v_j}$ are the edges of colour $c$ from $u$ to
$v_i$ and $v_j$ respectively.
By the earlier observation, no minimal path that starts with the edge
$e_{u v_i}$ of colour $c$ can pass through any of the vertices $v_j$ for $j<i$.
Hence, any minimal path that starts with the edge $e_{u v_{n-1}}$ has only a choice
of $k-1$ edges to reach $v$ (since its consecutive edges are not of the same colour),
so there are at most $k-1$ such paths.
Similarly, any minimal path that starts with the edge $e_{u v_i}$ can only
utilize $v_i, v_{i+1}, \dots, v_{n-1}$ and $v$, so by the inductive hypothesis for $m=n-i+1$
there are at most $\Sigma_{c^\prime\neq c}f_{c^\prime}(n-i+1)\leq(k-1)k^{n-i-1}$.
Summing across all the edges $e_{u v_1}, e_{u v_2}, \dots, e_{u v_{n-1}}$ and $e_{uv}$
gives
$f_c(n+1)\leq (k-1)k^{n-2} + (k-1)k^{n-3} + \cdots + (k-1) + 1 = k^{n-1}$.
Since there are $k$ possible colours in which to start a path this completes the proof.
\end{proof}

Note that as the proof relies on being able to linearly order the vertices
$v_1, v_2, \dots, v_{n-1}$
based on the edge weights $\omega\left(e_{u v_i}\right)$ in a specific colour $c$,
the proof can not be readily adapted to arbitrary multiweighted graphs.
This result gives a tight upper bound on the cardinality of a set of incomparable
minimal paths in a weighted coloured--edge graph.
The bound is suitable for applications that are primarily interested in
determining an optimal path given criteria that depend on the total weight in
each mode of transportation.
However, the proof can be sightly modified to provide a bound on the total number of
minimal paths in the graph from $u$ to $v$, counting all minimal paths that are
comparable with each other (having the same path weight).

\begin{theorem}\label{Thm:TotalMinimalPaths}
Suppose $\gra{G}$ is a weighted coloured--edge graph with $n\geq2$ vertices and $k$ colours
for which there is only at most one edge of each colour between vertices.
Then $\gra{G}$ has at most $k(k+1)^{n-2}$ minimal paths from one vertex to another.
\end{theorem}

\begin{proof}
Similar to Theorem \ref{Thm:IncomparableMinimalPaths} except that the counting argument
bounds $g_c(m)\leq (k+1)^{m-2}$ and paths are allowed to have the same colour
on two consecutive edges.
\end{proof}

The bound established in Theorem \ref{Thm:TotalMinimalPaths} can be seen to be
tight by constructing examples based on the chain example in the previous section
but with additional edges, such as the following example for $n=4$ in which
each of the $3\times4^2$ paths from $v_1$ to $v_4$ is minimal:

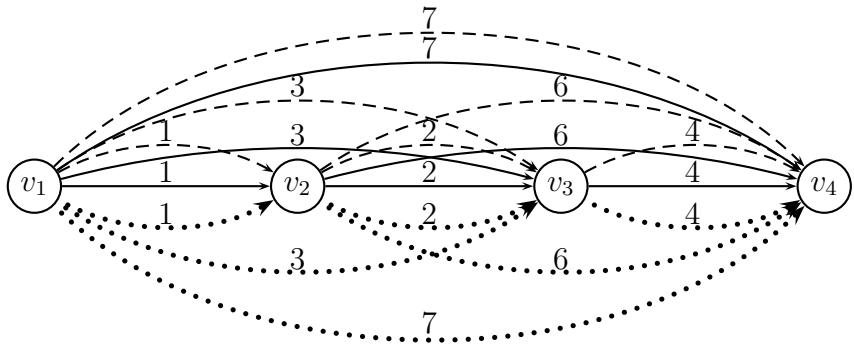
\begin{pspicture}(0,-2)(12,2.5)
  \rput(0.5,0){\circlenode{v1}{\makebox[4mm]{$v_1$}}}
  \rput(4,0){\circlenode{v2}{\makebox[4mm]{$v_2$}}}
  \rput(7.5,0){\circlenode{v3}{\makebox[4mm]{$v_3$}}}
  \rput(11,0){\circlenode{v4}{\makebox[4mm]{$v_4$}}}

  \nccurve[angleA=30,angleB=150,arrows=->,linestyle=dashed]{v1}{v2}
  \ncline[arrows=->]{v1}{v2}
  \nccurve[angleA=330,angleB=210,arrows=->,linestyle=dotted,linewidth=2pt]{v1}{v2}
  \rput[b](2.25,0.6){$1$}\rput[b](2.25,0.05){$1$}\rput[b](2.25,-0.5){$1$}

  \nccurve[angleA=30,angleB=150,arrows=->,linestyle=dashed]{v2}{v3}
  \ncline[arrows=->]{v2}{v3}
  \nccurve[angleA=330,angleB=210,arrows=->,linestyle=dotted,linewidth=2pt]{v2}{v3}
  \rput[b](5.75,0.6){$2$}\rput[b](5.75,0.05){$2$}\rput[b](5.75,-0.5){$2$}

  \nccurve[angleA=30,angleB=150,arrows=->,linestyle=dashed]{v3}{v4}
  \ncline[arrows=->]{v3}{v4}
  \nccurve[angleA=330,angleB=210,arrows=->,linestyle=dotted,linewidth=2pt]{v3}{v4}
  \rput[b](9.25,0.6){$4$}\rput[b](9.25,0.05){$4$}\rput[b](9.25,-0.5){$4$}

  \nccurve[angleA=35,angleB=145,arrows=->,linestyle=dashed]{v1}{v3}
  \nccurve[angleA=15,angleB=165,arrows=->]{v1}{v3}
  \nccurve[angleA=325,angleB=215,arrows=->,linestyle=dotted,linewidth=2pt]{v1}{v3}
  \rput[b](4,1.2){$3$}\rput[b](4,0.55){$3$}\rput[b](4,-1.1){$3$}

  \nccurve[angleA=35,angleB=145,arrows=->,linestyle=dashed]{v2}{v4}
  \nccurve[angleA=15,angleB=165,arrows=->]{v2}{v4}
  \nccurve[angleA=325,angleB=215,arrows=->,linestyle=dotted,linewidth=2pt]{v2}{v4}
  \rput[b](7.5,1.2){$6$}\rput[b](7.5,0.55){$6$}\rput[b](7.5,-1.1){$6$}

  \nccurve[angleA=45,angleB=135,arrows=->,linestyle=dashed]{v1}{v4}
  \nccurve[angleA=35,angleB=145,arrows=->]{v1}{v4}
  \nccurve[angleA=315,angleB=225,arrows=->,linestyle=dotted,linewidth=2pt]{v1}{v4}
  \rput[b](5.75,2.1){$7$}\rput[b](5.75,1.7){$7$}\rput[b](5.75,-1.95){$7$}
\end{pspicture}

\section{Expected Number of Minimal Paths}

This section demonstrates that the expected number of minimal paths for a
bicoloured--edge graph is polynomially bounded.
The approach is based on some ideas from R\"{o}glin \& V\"{o}cking \cite{Roglin2005}
and Beier et al. \cite{heiko2007} where a bound on the expected number of
optimal solutions is estimated for bicriteria problems.
Their work focuses on the establishment of a probabilistic bound on the number
of Pareto optimal points for bicriteria integer problems,
exploiting structural properties of the Pareto frontier
that are termed as \defn{winners} and \defn{losers}.
It is adapted here to estimate a probabilistic bound on the number of minimal paths
in weighted bicoloured--edge graphs.
However, several arguments have been modified to be applied in the context of
coloured--edge graphs.

Suppose $\gra{G}$ is a finite weighted bicoloured--edge graph with colours
$M=\{\mbox{red}, \mbox{green}\}$ for convenience and let $u$, $v$ be vertices
of $\gra{G}$ for which there is a pure coloured--edge path in colour green from $u$ to $v$.

Assume $e$ is a red edge of $\gra{G}$ whose weight is a random variable with
bounded probability density function $f_e:(0,\infty)\rightarrow[0,\phi_e]$ for some $\phi_e>0$.
Define the function $\Delta_e:[0,\infty)\rightarrow(0,\infty]$ for $r\geq 0$ by the following.
Consider the paths $p_e$ from $u$ to $v$ that do not include the edge $e$ and for which
$\omega_{\mbox{\scriptsize red}}(p_e)\le r$. Since there is a pure path in colour green from $u$ to $v$,
there are such paths $p_e$, and since $\gra{G}$ is finite there are only finitely  many
such paths. Take $p_e^{max}$ to be such a path that has least green
weight and let $g_r = \omega_{\mbox{\scriptsize green}}(p_e^{max})$. Note that $g_r$ is uniquely
defined for $r$ and does not depend in any way on the value of $\omega(e)$.
Next, consider the paths $q_e$ from $u$ to $v$ that do include the edge $e$ and for which
$\omega_{\mbox{\scriptsize green}}(q_e) < g_r$. If there is no such path then take
$\Delta_e(r)=\infty$ for convenience, otherwise let $q_e^{min}$ denote such a path that
has least red weight and take $\Delta_e(r)=\omega_{\mbox{\scriptsize red}}(q_e^{min})$.
Note that although $\omega_{\mbox{\scriptsize red}}(q_e^{min})$ depends on the value of $\omega(e)$,
the weight of this edge does not affect the relative red ordering between
the various $q_e$ (since they each include $e$). Hence the choice of $q_e^{min}$
(or another $q_e$ with same red weight) does not depend in any way
on the value of $\omega(e)$, and $s_r=\omega_{\mbox{\scriptsize red}}(q_e^{min})-\omega(e)$ (where
$s_r$ is the sum of red weights except for $e$) is uniquely determined by $r$
and does not depend on the choice of $\omega(e)$.

\begin{figure}[h]
\begin{center}
\scalebox{0.65}{
\begin{pspicture}(-0.5,-0.5)(12,12)
  \psdot[dotsize=0.3](0,10.5)
  \psdot[dotsize=0.3](1.5,8.5)
  \psdot[dotsize=0.3](2.5,6.5)
  \psdot[dotsize=0.3,dotstyle=o](2.5,9.5)
  \psdot[dotsize=0.3](3.5,7.5)
  \psdot[dotsize=0.3](4,9)
  \psdot[dotsize=0.3](4,10)
  \psdot[dotsize=0.3](4.5,6.5)
  \psdot[dotsize=0.3](5,5)\uput{1.5mm}[225](5,5){$p_e^{max}$}
  \psdot[dotsize=0.3,dotstyle=o](5,7.5)
  \psdot[dotsize=0.3,dotstyle=o](5.5,3.5)\uput{1.5mm}[225](5.5,3.5){$q_e^{min}$}
  \psdot[dotsize=0.3](5.5,6)
  \psdot[dotsize=0.3,dotstyle=o](6.5,7.5)
  \psdot[dotsize=0.3](7,10)
  \psdot[dotsize=0.3,dotstyle=o](7.5,3)
  \psdot[dotsize=0.3](7.5,4.5)
  \psdot[dotsize=0.3](7.5,6.5)
  \psdot[dotsize=0.3,dotstyle=o](8,8)
  \psdot[dotsize=0.3](8.5,5.5)
  \psdot[dotsize=0.3,dotstyle=o](9,4.5)
  \psdot[dotsize=0.3](9.5,2.5)
  \psdot[dotsize=0.3](10,8)
  \psdot[dotsize=0.3](10.5,3.5)
  \psdot[dotsize=0.3,dotstyle=o](10.5,9)
  \psdot[dotsize=0.3](11,0)

  \pcline[linewidth=0.02,arrowsize=0.3]{<->}(0,1)(5.5,1)\ncput*[framesep=2pt]{$\Delta_e(r)$}
  \pcline[linestyle=dotted,linewidth=.07](5.5,0)(5.5,3.4)
  \pcline[linestyle=dashed,linewidth=.01](6,0)(6,12)\uput{2mm}[270](6,0){$r$}
  \pcline[linestyle=dashed,linewidth=.01](0,5)(12,5)\uput{2mm}[180](0,5){$g_{r}$}

  \psline[linewidth=0.06]{<->}(0,12)(0,0)(12,0)
  \uput{5mm}[270](12,0){red}
  \uput{7mm}[180](0.5,12){green}
\end{pspicture}}
\end{center}
\caption{Representation of $\Delta_e(r)$ and associated variables.}
\label{Fig:DeltaE}
\end{figure}
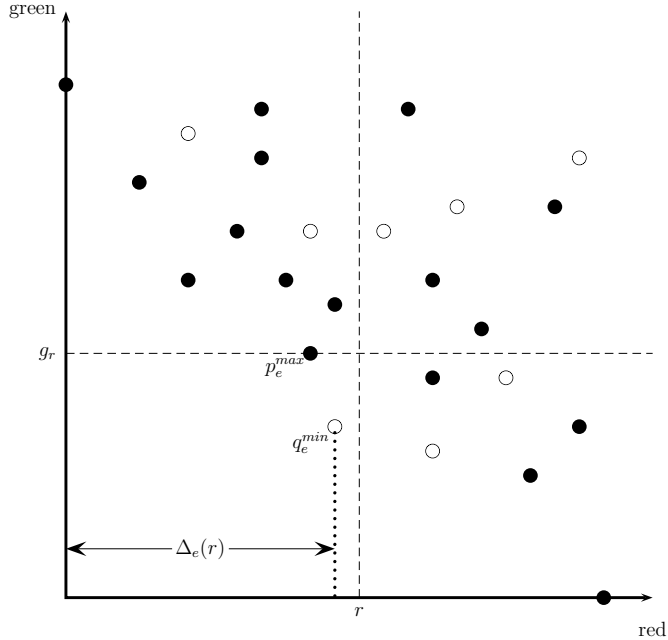

Figure \ref{Fig:DeltaE} illustrates $\Delta_e$ and its associated variables,
where black dots are used to represent paths not using edge $e$ and white dots represent
the paths containing $e$.
Note that all $q_e$ paths just shift horizontally depending on the value of $\omega(e)$.
However, they do not change their positions relative to each other.

\begin{lemma}\label{Lem:Delta1}
For any $r\geq 0$ and $\varepsilon \geq 0$, if $\Delta_e(r)<\infty$ then
\begin{displaymath}
  \mathbb{P}(r < \Delta_e(r) \leq r+\varepsilon) \leq \phi_e \cdot \varepsilon.
\end{displaymath}
\end{lemma}

\begin{proof}
Let $r\geq 0$, $\varepsilon > 0$ and $q_e^{min}$ be a path that includes $e$ with
$\omega_{\mbox{\scriptsize green}}(q_e^{min}) < g_r$ and $\omega_{\mbox{\scriptsize red}}(q_e^{min})$
minimal amongst such paths.
As $s_r$ does not depend on the value of $\omega(e)$ one has:
\begin{eqnarray*}
\mathbb{P}(r < \Delta_e(r) \leq r+\varepsilon)
  &=& \mathbb{P}(r < s_r + \omega(e) \leq r+\varepsilon)\\
  &=& \mathbb{P}(r - s_r < \omega(e) \leq r - s_r + \varepsilon)\\
  &=& \int_{r-s_r}^{r-s_r+\varepsilon} f_e(x)dx \\
  &\leq& \int_{r-s_r}^{r-s_r+\varepsilon} \phi_e dx = \phi_e \cdot \varepsilon.
\end{eqnarray*}
\end{proof}

Now, suppose that all the red edges $e$ of the graph $\gra{G}$ have weights that are random
variables with bounded probability density functions, and suppose that besides a pure green
path from $u$ to $v$ there is also a pure coloured--edge path in red from $u$ to $v$,
with a minimal pure coloured--edge path in red having red weight $r_{tot}$.

Define the function $\Delta:[0,r_{tot})\rightarrow(0,\infty)$ for $0\leq r < r_{tot}$
by the following. Consider the minimal paths $q$ from $u$ to $v$ for
which $\omega_{\mbox{\scriptsize red}}(q)>r$.
Since there is a pure red path with weight $r_{tot}$, there are such paths $q$, and since
$\gra{G}$ is finite there are only finitely many such paths. Take $q^{min}$ to be a minimal
path with $\omega_{\mbox{\scriptsize red}}(q^{min})>r$ that has least red weight and take
$\Delta(r)=\omega_{\mbox{\scriptsize red}}(q^{min})>r$.
Figure \ref{Fig:Delta01} illustrates $\Delta(r)$ and $q^{min}$.
\begin{figure}[h]
\begin{center}
\scalebox{0.65}{
\begin{pspicture}(-0.5,-0.5)(12,12)
  \psdot[dotsize=0.3](0,10.5)
  \psdot[dotsize=0.3](1.5,8.5)
  \psdot[dotsize=0.3](2.5,6.5)
  \psdot[dotsize=0.3](2.5,9.5)
  \psdot[dotsize=0.3](3.5,7.5)
  \psdot[dotsize=0.3](4,9)
  \psdot[dotsize=0.3](4,10)
  \psdot[dotsize=0.3](4.5,6.5)
  \psdot[dotsize=0.3](5,5)
  \psdot[dotsize=0.3](5,7.5)
  \psdot[dotsize=0.3](5.5,3.5)\uput{1.75mm}[225](5.5,3.5){$p^{max}$}
  \psdot[dotsize=0.3](5.5,6)
  \psdot[dotsize=0.3](6.5,7.5)
  \psdot[dotsize=0.3](7,10)
  \psdot[dotsize=0.3](7.5,3)\uput{1.5mm}[225](7.5,3){$q^{min}$}
  \psdot[dotsize=0.3](7.5,4.5)
  \psdot[dotsize=0.3](7.5,6.5)
  \psdot[dotsize=0.3](8,8)
  \psdot[dotsize=0.3](8.5,5.5)
  \psdot[dotsize=0.3](9,4.5)
  \psdot[dotsize=0.3](9.5,2.5)
  \psdot[dotsize=0.3](10,8)
  \psdot[dotsize=0.3](10.5,3.5)
  \psdot[dotsize=0.3](10.5,9)
  \psdot[dotsize=0.3](11,0)

  \psline[linestyle=dotted,linewidth=.05](0,10.5)(1.5,8.5)(2.5,6.5)(5,5)%
    (5.5,3.5)(7.5,3)(9.5,2.5)(11,0)
  \uput{1mm}[0](10.25,1.25){Pareto frontier}

  \pcline[linewidth=0.02,arrowsize=0.3]{<->}(0,1)(7.5,1)\ncput*[framesep=2pt]{$\Delta(r)$}
  \pcline[linestyle=dotted,linewidth=.07](7.5,0)(7.5,3)
  \pcline[linestyle=dashed,linewidth=.01](6,0)(6,12)\uput{2mm}[270](6,0){$r$}

  \psline[linewidth=0.06]{<->}(0,12)(0,0)(12,0)
  \uput{5mm}[270](12,0){red}
  \uput{7mm}[180](0.5,12){green}
\end{pspicture} }
\end{center}
\caption{Representation of $\Delta(r)$ and associated variables.}
\label{Fig:Delta01}
\end{figure}
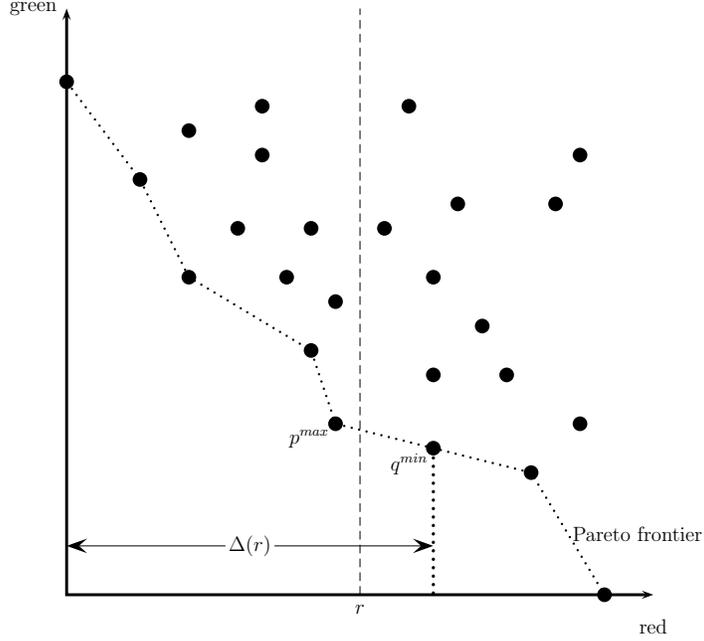

\begin{lemma}\label{Lem:Delta2}
For $0 \leq r < r_{tot}$, there exists a red edge $e$ for which $\Delta(r)=\Delta_e(r)$.
\end{lemma}

\begin{proof}
Consider the minimal paths $p$ from $u$ to $v$ for which $\omega_{\mbox{\scriptsize red}}(p)\leq r$.
Since there is a pure green path from $u$ to $v$ there are such paths, and
since $\gra{G}$ is finite there are only finitely many such paths.
Take $p^{max}$ to be a minimal path with $\omega_{\mbox{\scriptsize red}}(p)\le r$
that has least green weight.
Then $p^{max}$ and $q^{min}$ are adjacent minimal paths on the Pareto frontier,
so there can be no path between $p^{max}$ and $q^{min}$ in the sense that no path
can have both its red weight less than
$\omega_{\mbox{\scriptsize red}}(q^{min})$ and its green weight less
than $\omega_{\mbox{\scriptsize green}}(p^{max})$.

Since $\omega_{\mbox{\scriptsize red}}(p^{max}) \leq r < \omega_{\mbox{\scriptsize red}}(q^{min})$
there must be
some red edge $e$ that is in $q^{min}$ but not in $p^{max}$.
As $p_e^{max}$ has the least green weight amongst paths
$p_e$ that do not include $e$ and for which $\omega_{\mbox{\scriptsize red}}(p_e)\le r$,
$\omega_{\mbox{\scriptsize green}}(p_e^{max}) \leq \omega_{\mbox{\scriptsize green}}(p^{max})$.
As
$\omega_{\mbox{\scriptsize red}}(p_e^{max}) \leq r < \omega_{\mbox{\scriptsize red}}(q^{min})$,
and there are no paths between $p^{max}$ and $q^{min}$ it follows that
$\omega_{\mbox{\scriptsize green}}(p_e^{max})=\omega_{\mbox{\scriptsize green}}(p^{max})$.
Hence $g_r = \omega_{\mbox{\scriptsize green}}(p^{max})$.

As $p^{max}$ and $q^{min}$ are incomparable
$\omega_{\mbox{\scriptsize green}}(q^{min})<\omega_{\mbox{\scriptsize green}}(p^{max})=g_r$.
Next, as $q_e^{min}$ has the least red weight amongst paths $q_e$
that do include $e$ and for which $\omega_{\mbox{\scriptsize green}}(q_e)<g_r$,
$\omega_{\mbox{\scriptsize red}}(q_e^{min}) \le \omega_{\mbox{\scriptsize red}}(q^{min})$.
But since
$\omega_{\mbox{\scriptsize green}}(q_e^{min}) < g_r = \omega_{\mbox{\scriptsize green}}(p^{max})$
and there are no paths between $p^{max}$ and $q^{min}$, it follows
that $\omega_{\mbox{\scriptsize red}}(q_e^{min})=\omega_{\mbox{\scriptsize red}}(q^{min})$.
Hence one has
$\Delta(r)=\omega_{\mbox{\scriptsize red}}(q^{min})=\omega_{\mbox{\scriptsize red}}(q_e^{min})=\Delta_e(r)$.
\end{proof}

\begin{corollary}\label{Cor:Delta}
For any $0 \le r < r_{tot}$ and $\varepsilon > 0$
\begin{displaymath}
  \mathbb{P}(\Delta(r) \le r + \varepsilon)
  \le \left(\sum_{\mbox{\scriptsize red edge $e$}} \phi_e \right) \cdot \varepsilon.
\end{displaymath}
\end{corollary}
\begin{proof}
If $ r < \Delta(r) \le r + \varepsilon$ then by Lemma \ref{Lem:Delta2}, there exists a
red edge $e$ for which $ r < \Delta_e(r) \le r + \varepsilon$. Hence using a union
bound and Lemma \ref{Lem:Delta1},
\begin{eqnarray*}
  \mathbb{P}(r < \Delta(r) \le r + \varepsilon)
  &\le& \sum_{\mbox{\scriptsize red edge $e$}} \mathbb{P}(r < \Delta_e(r) \le r + \varepsilon)\\
  &\le& \sum_{\mbox{\scriptsize red edge $e$}} \phi_e \cdot \varepsilon.
\end{eqnarray*}
\end{proof}

Corollary \ref{Cor:Delta} provides the main argument to establish a
bound on the expected number of minimal paths for a bicoloured--edge graph.

\begin{theorem}\label{Thm:ExpectedPareto}
Let $\gra{G}$ be a finite weighted bicoloured--edge graph and let $u$, $v$ be vertices of $\gra{G}$
for which there is a pure colour path from $u$ to $v$ in each of the two colours.
Suppose that the weights of all edges $e$ in one of the colours are random variables
with probability density functions bounded above by $\phi_e$, and let $r_{tot}$ denote
the weight of the minimal pure colour path in that colour.
Then the expected number of Pareto minimal elements with distinct weights is bounded
above by $\sum_e \phi_e \cdot r_{tot} +1$.
\end{theorem}

\begin{proof}
As previously denote the colours by $\{\mbox{red},\mbox{green}\}$ for convenience.
As $\gra{G}$ is finite there are only finitely many minimal paths $q$
from $u$ to $v$, and they all have red weight between $0$ and $r_{tot}$
inclusive.

Partition the interval $(0,r_{tot}]$ into $\kappa$ equal subintervals
and note that since only minimal paths with distinct (red) weights
are considered, there is a threshold $\kappa_{min}$ above which each
interval can contain at most one minimal path. Hence for all $\kappa \ge \kappa_{min}$,
the expected number of distinct minimal paths is
\[1+\sum_{i=0}^{\kappa-1}\mathbb{P}\left(\exists \; \mbox{minimal path} \;q \; \mbox{with}\;\frac{r_{tot}}{\kappa}i<\omega_{\mbox{\scriptsize red}}(q)\leq\frac{r_{tot}}{\kappa}(i+1)\right)\]
including the minimal pure green path that has red weight $0$.
Figure \ref{Fig:Delta03} depicts the partition of $(0,r_{tot}]$.

\begin{figure}[h]
\begin{center}
\scalebox{0.65}{
\begin{pspicture}(-0.5,-0.5)(12,12)
  \psdot[dotsize=0.3](0,10.5)
  \psdot[dotsize=0.3](1.5,8.5)
  \psdot[dotsize=0.3](2.5,6.5)
  \psdot[dotsize=0.3](5,5)
  \psdot[dotsize=0.3](5.5,3.5)
  \psdot[dotsize=0.3](7.5,3)
  \psdot[dotsize=0.3](9.5,2.5)
  \psdot[dotsize=0.3](11,0)\uput{2mm}[270](11,0){$r_{tot}$}

  \pcline[linestyle=dotted,linewidth=.05](0.46,0)(0.46,12)
  \pcline[linestyle=dotted,linewidth=.05](0.92,0)(0.92,12)
  \pcline[linestyle=dotted,linewidth=.05](1.38,0)(1.38,12)
  \pcline[linestyle=dotted,linewidth=.05](1.84,0)(1.84,12)
  \pcline[linestyle=dotted,linewidth=.05](2.3,0)(2.3,12)
  \pcline[linestyle=dotted,linewidth=.05](2.76,0)(2.76,12)
  \pcline[linestyle=dotted,linewidth=.05](3.22,0)(3.22,12)
  \pcline[linestyle=dotted,linewidth=.05](3.68,0)(3.68,12)
  \pcline[linestyle=dotted,linewidth=.05](4.14,0)(4.14,12)
  \pcline[linestyle=dotted,linewidth=.05](4.6,0)(4.6,12)
  \pcline[linestyle=dotted,linewidth=.05](5.06,0)(5.06,12)
  \pcline[linestyle=dotted,linewidth=.05](5.52,0)(5.52,12)
  \pcline[linestyle=dotted,linewidth=.05](5.98,0)(5.98,12)
  \pcline[linestyle=dotted,linewidth=.05](6.44,0)(6.44,12)
  \pcline[linestyle=dotted,linewidth=.05](6.9,0)(6.9,12)
  \pcline[linestyle=dotted,linewidth=.05](7.36,0)(7.36,12)
  \pcline[linestyle=dotted,linewidth=.05](7.82,0)(7.82,12)
  \pcline[linestyle=dotted,linewidth=.05](8.28,0)(8.28,12)
  \pcline[linestyle=dotted,linewidth=.05](8.74,0)(8.74,12)
  \pcline[linestyle=dotted,linewidth=.05](9.2,0)(9.2,12)
  \pcline[linestyle=dotted,linewidth=.05](9.66,0)(9.66,12)
  \pcline[linestyle=dotted,linewidth=.05](10.12,0)(10.12,12)
  \pcline[linestyle=dotted,linewidth=.05](10.58,0)(10.58,12)
  \pcline[linestyle=dotted,linewidth=.05](11.04,0)(11.04,12)

  \pcline[linewidth=.02]{<->}(0,-0.25)(0.45,-0.25)
  \uput{0mm}[-90](0.225,-0.4){$\frac{r_{tot}}{\kappa}$}

  \psline[linewidth=0.06]{<->}(0,12)(0,0)(12,0)
  \uput{5mm}[270](12,0){red}
  \uput{7mm}[180](0.5,12){green}
\end{pspicture} }
\end{center}
\caption{Pareto minimal elements and partition of the interval $(0,r_{tot}]$.}
\label{Fig:Delta03}
\end{figure}
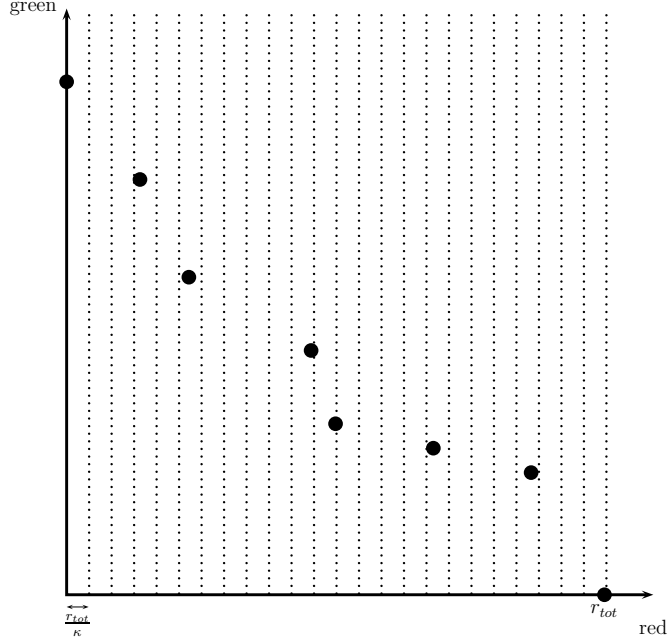

\noindent Now for each $i$ there is a minimal path $q$ with
$\frac{r_{tot}}{\kappa}i<\omega_{\mbox{\scriptsize red}}(q)\leq\frac{r_{tot}}{\kappa}(i+1)$
if and only if $\Delta(\frac{r_{tot}}{\kappa}i)\leq \frac{r_{tot}}{\kappa}(i+1)$
which has probability bounded above by $(\sum \phi_e) \frac{r_{tot}}{\kappa}$
by Corollary \ref{Cor:Delta}. Hence the expected number of minimal paths with distinct
weights is bounded by
\begin{displaymath}
  1+\sum_{i=0}^{\kappa-1}\left(\sum \phi_e\right) \frac{r_{tot}}{\kappa}
  = 1+ (\sum \phi_e)\cdot r_{tot}.
\end{displaymath}
\end{proof}

Note that if each red edge has weight bounded by $r_{max}$ then $r_{tot}$ is
$O(n)$, so the expected number of minimal paths is $O(mn)+1$ where $m$ denotes
the number of red edges. Furthermore, there are only at most $n^2-3n+3$ red edges
to consider in a graph with $n$ vertices so the order is bounded by $O(n^3)$.

\section{Conclusions}

For weighted coloured--edge graphs to be successfully utilized
in multimodal network applications the number of minimal paths needs to be
manageable, as each minimal path may need to be further investigated to determine
a sought optimal path or paths in a particular application.
The chain example from Section \ref{Sec:Introduction} with exponentially growing
weights illustrates that there can be an exponential number of incomparable minimal
paths in a weighted coloured--edge graph.
However, Theorem \ref{Thm:IncomparableMinimalPaths} shows that this is a tight bound,
and the fact that there can only be an exponential number $k^{n-1}$ of minimal paths
rather than potentially a factorial number $O\left(k^{n-1}(n-2)!\right)$ is surprising.
As a consequence, even in the worst case moderately sized graphs (with around $20-30$
vertices) can still be feasibly tackled.

Experimental studies undertaken by the authors indicate that the number of minimal
paths in real networks is typically a low-order polynomial function of $n$,
so very large networks can be studied in practice.
Theorem \ref{Thm:ExpectedPareto} justifies this observation for bicoloured--edge
graphs whose edge weights are randomly drawn from a bounded
probability density function,
showing that only $O\left(n^3\right)$ minimal paths are expected.
It is presumed that this polynomial bound can be substantially reduced and that a similar result
is also true for coloured--edge graphs with $k>2$ colours.



\section*{Acknowledgements}
This research was partly supported by Catolica del Maule
University, Talca--Chile, through the project MECESUP--UCM0205.

\bibliographystyle{elsarticle-num}
\bibliography{cocoa}{}

\end{document}